\documentclass[preprint,12pt]{elsarticle}
\usepackage{amssymb}
\usepackage{lineno,hyperref}
\usepackage{amsmath}
\usepackage{amsthm}
\usepackage{graphicx}
\newtheorem{lemma}{Lemma}
\newtheorem{theorem}{Theorem}
\newtheorem{remark}{Remark}

\modulolinenumbers[5]

\journal{}

\begin{document}

\begin{frontmatter}



\title{Existence of smooth solutions of multi-term Caputo-type fractional differential equations}

\author[Firstaddress,Secondaddress]{Chung-Sik Sin}
\cortext[mycorrespondingauthor]{Corresponding author}
\ead{chongsik@163.com}

\author[Firstaddress]{Shusen Cheng}
\author[Secondaddress]{Gang-Il Ri}
\author[Secondaddress]{Mun-Chol Kim}

\address[Firstaddress]{School of Metallurgy and Ecology Engineering, University of Science and Technology Beijing, Beijing, 30Xueyuan Road, Haidian District,  China}
\address[Secondaddress]{Faculty of Mathematics, \textit{\textbf {Kim Il Sung}} University, Kumsong Street, Taesong District, Pyongyang, D.P.R.KOREA}

\begin{abstract}
This paper deals with the initial value problem for the multi-term fractional differential equation. The fractional derivative is defined in the Caputo sense. Firstly the initial value problem is transformed into a equivalent Volterra-type integral equation under appropriate assumptions.
Then new existence results for smooth solutions are established by using the Schauder fixed point theorem.
\end{abstract}

\end{frontmatter}


\section{Introduction}\label{sec:1}

 In the present paper, we consider multi-term Caputo-type fractional differential equations of the form
 \begin{equation}\label{multi_term_equation}
 D^{\alpha_n} u(t) =f(t,D^{\alpha_1}u(t),...,D^{\alpha_{n-1}} u(t))
 \end{equation}
 subject to initial conditions
 \begin{equation}\label{multi_term_condition}
 u^{(i)}(0)=u_0^{(i)},  i=0,1,...,\lceil \alpha_n \rceil-1,
 \end{equation}
 where $ \alpha_n>\alpha_{n-1}>...>\alpha_1 \geq 0 $ and the symbol $ D^{\beta} $ denotes the Caputo-type fractional differential operator defined by (\cite{DieBoo,KilSri})
 \begin{equation}
 \nonumber
 D^\beta u=J^{\lceil \beta \rceil-\beta}u^{(\lceil \beta \rceil)}.
 \end{equation}
 Here $ J^\gamma $ is the Riemann-Liouville integral operator of order $ \gamma \geq 0 $ defined by $ J^0 $ being the identity operator and
 \begin{equation}
 \nonumber
 J^\delta u(t)={\frac{1}{\Gamma (\delta) } \int^{t}_{0}{(t-s)}^{\delta-1} u(s) ds}
 \end{equation}
 for $ \delta>0 $.

 Fractional differential equations have excited a great deal of interest in areas such as porous media, plasma dynamics, thermodynamics, cosmic rays, continuum mechanics, biological systems, electrodynamics, quantum mechanics (see \cite{Hilfer,KilSri,Mainardi_book,Mai,Pod,Sin2,Uchaikin}). Especially the relaxation modulus and creep compliance of the multi-term fractional constitutive model proposed in \cite{Bagley} which describe the linear viscoelastic behaviour are obtained from multi-term fractional differential equations \cite{Mainardi_book}.

 In general, it is very difficult to obtain analytical solutions of fractional differential equations.
 Although an analytical expression of solutions of initial value problems of linear differential equations with constant coefficients and Caputo derivative is given in the paper \cite{LucGor}, it is quite cumbersome to handle.
 The predictor-corrector method is one of powerful tools for obtaining numerical solutions of fractional differential equations (see \cite{DieFor1,DieFor3,YangLiu}).
 Diethelm et al. \cite{DieFor2} and Yang et al. \cite{YangLiu} proved that the initial value problem (\ref{multi_term_equation})-(\ref{multi_term_condition}) is equivalent to a fractional differential system when the solution $ u $ is in $ C^{\lceil \alpha_n \rceil}[0,L] $. Based on the equivalence theorems, the authors used the predictor-corrector method to obtain numerical solutions of the initial value problem (\ref{multi_term_equation})-(\ref{multi_term_condition}).
  Thus it is significant to study the existence of $ \lceil \alpha_n \rceil  $ times continuously differentiable solutions of the initial value problem (\ref{multi_term_equation})-(\ref{multi_term_condition}).

 Analytical properties of fractional differential equations can be investigated by considering equivalent Volterra-type integral equations
 (see \cite{DieBoo,DieSmo,DieFor1,KilSri,Kos,Sin}).
 Diethelm et al. \cite{DieSmo} studied the smoothness of solutions of single-term Caputo-type differential equations.
 The paper \cite{Kos} studied the existence of  $ \lceil \alpha_2 \rceil $ times continuously differentiable solutions to two-term fractional differential equations of the form
 \begin{equation}\label{two_term_equation}
 D^{\alpha_2} u(t) =f(t,D^{\alpha_1}u(t))
 \end{equation}
 subject to initial conditions
 \begin{equation}\label{two_term_condition}
 u^{(i)}(0)=u_0^{(i)},  i=0,1,...,\lceil \alpha_2 \rceil-1.
 \end{equation}
 The following lemmas is essential in \cite{Kos}.

 \begin{lemma} [\cite{Kos}]\label{incorrect_lemma1}
 Let $  \alpha_1, \alpha_2 \notin N $ and $ \lceil \alpha_1 \rceil  < \lceil \alpha_2 \rceil $. Suppose that $ f(0,0)=0, f(t,0) \neq 0 $ on a compact subinterval of $(0,1]$ and $ f:[0,1]\times R \rightarrow R $ is continuously differentiable.
 Then a function $ u \in C^{\lceil \alpha_2 \rceil}[0,1] $ is a solution of the initial value problem (\ref{two_term_equation})-(\ref{two_term_condition}) if and only if
 \begin{equation} \label{incorrect_lemma1_eq1}
 u(t)=\sum^{ \lceil \alpha_1 \rceil -1}_{i=0}\frac{t^i}{i!}u^{(i)}_0
 +\int^t_0 \frac{{(t-s)}^{ \lceil \alpha_1 \rceil-1}}{(\lceil \alpha_1 \rceil -1)!}v(s)ds,
 t \in [0,1],
 \end{equation}
 where $ v \in C[0,1] $ is a solution of the integral equation
 \begin{align}
 \nonumber
 v(t)&=\sum^{\lceil \alpha_2 \rceil-\lceil \alpha_1 \rceil-1}_{i=0}\frac{t^i}{i!}u^{(i+\lceil \alpha_1 \rceil)}_0+\int^t_0 \frac {{(t-s)}^{ \alpha_2-\lceil \alpha_1 \rceil-1}}{\Gamma(\alpha_2-\lceil \alpha_1 \rceil)}
 f\biggl (s,\frac {1}{\Gamma(\lceil \alpha_1 \rceil-\alpha_1)}\\
 \label{incorrect_lemma1_eq2}
 &\int^s_0 {(s-w)}^{\lceil \alpha_1 \rceil-\alpha_1-1}v(w)dw \biggr)ds.
 \end{align}
 \end{lemma}

 \begin{lemma}[\cite{Kos}] \label{incorrect_lemma2}
 Let $  \alpha_1, \alpha_2 \notin N $ and $ \lceil \alpha_1 \rceil  = \lceil \alpha_2 \rceil $. Suppose that $ f(0,0)=0, f(t,0) \neq 0 $ on a compact subinterval of $(0,1]$ and $ f:[0,1]\times R \rightarrow R $ is continuously differentiable.
 Then a function $ u \in C^{\lceil \alpha_2 \rceil}[0,1] $ is a solution of the initial value problem (\ref{two_term_equation})-(\ref{two_term_condition}) if and only if
 \begin{equation}\label{incorrect_lemma2_eq1}
 u(t)=\sum^{ \lceil \alpha_1 \rceil -1}_{i=0}\frac{t^i}{i!}u^{(i)}_0
 +\frac{1}{\Gamma(\alpha_1)} \int^t_0 {(t-s)}^{ \alpha_1 -1}v(s)ds, t \in [0,1],
 \end{equation}
 where $ v \in C[0,1] $ is a solution of the integral equation
 \begin{equation}\label{incorrect_lemma2_eq2}
  v(t)=\frac{1}{\Gamma(\alpha_2- \alpha_1)}\int^t_0 {(t-s)}^{ \alpha_2- \alpha_1 -1}f(s,v(s)).
 \end{equation}
 \end{lemma}
 Based on Lemma \ref{incorrect_lemma1} and Lemma \ref{incorrect_lemma2}, the authors \cite{Deng2,DengMa} established the existence and uniqueness of solutions to the initial value problem  (\ref{two_term_equation})-(\ref{two_term_condition}) on the interval [0,1].

 The present paper is organized as follows.
 In Section 2, we transform the initial value problem (\ref{multi_term_equation})- (\ref{multi_term_condition}) into a equivalent Volterra-type integral equation under proper assumptions.
 In particular, it is prove that Lemma \ref{incorrect_lemma1} and Lemma \ref{incorrect_lemma2} are incorrect.
 With the help of Section 2, the correct existence results for smooth solutions to the initial value problem (\ref{multi_term_equation})- (\ref{multi_term_condition}) are established in Section 3.


\section{Equivalent integral equations}\label{sec:2}

 In this section it is proved that the solvability of the initial value problem (\ref{multi_term_equation})- (\ref{multi_term_condition}) is equivalent to that of a Volterra-type integral equation.
 \begin{lemma}[\cite{LiDeng}]\label{le2-1}
 Let $ I>0 $ and assume that $ u \in C^m[0,I] $ and $ m-1<\beta<\rho <m $. Then, for all k $ \in\{1,...,m-1\} $,
 $ D^{\rho-m+k}u^{(m-k)}(t)=D^{\rho}u(t)$ and $D^{\rho-\beta}{D^{\beta}}u(t)=D^{\rho}u(t)$.
 \end{lemma}
 The following lemma plays an important role in our consideration.
 \begin{lemma}\label{le2-2}
 Let $ I>0 $, $ 0< \beta <1 $, $ g \in C^1[0,I] $ and the function $ F:[0,I] \rightarrow R $
 is defined by
 \begin{equation}
 \nonumber
 F(t)=\int^t_0 {(t-s)}^{-\beta}g(s)ds.
 \end{equation}
 Then $ F(t) \in C^1[0,I] $ if and only if $ g(0)=0 $.
 \end{lemma}
 \begin{proof}
 \begin{equation}
 \nonumber
 F(t)=\frac{t^{1-\beta}}{1-\beta}g(0)+\frac{1}{1-\beta}\int^t_0 {(t-s)}^{1-\beta}g^{\prime}(s)ds, t \in [0,I].\\
 \end{equation}
 \begin{equation}
 \nonumber
 F^\prime(t)=t^{-\beta}g(0)+\int^t_0 {(t-s)}^{-\beta}g^{\prime}(s)ds, t \in (0,I].
 \end{equation}
 Thus, $ F^\prime(t) $ is continuous in $ [0,I] $ if and only if $ g(0)=0 $.
 \end{proof}

 \begin{theorem}\label{th2-1}
 Let $ I>0 $, $ \alpha_n, \alpha_{n-1} \notin N $ and $ \lceil \alpha_{n-1} \rceil+1  < \alpha_n $. Suppose that $ f(0,0)=0 , u^{(\lceil \alpha_{n-1} \rceil)}_0=0 $ and $ f:[0,I]\times R \rightarrow R $ is continuously differentiable.
 Then a function $ u \in C^{\lceil \alpha_n \rceil}[0,I] $ is a solution of the initial value problem (\ref{multi_term_equation})- (\ref{multi_term_condition}) if and only if
 \begin{equation}\label{th2-1_eq1}
 u(t)=\sum^{ \lceil \alpha_{n-1} \rceil -1}_{i=0}\frac{t^i}{i!}u^{(i)}_0
 +\int^t_0 \frac{{(t-s)}^{ \lceil \alpha_{n-1} \rceil-1}}{(\lceil \alpha_{n-1} \rceil -1)!}v(s)ds,
 t \in [0,I],
 \end{equation}
 where $ v \in C[0,I] $ is a solution of the integral equation
 \begin{align}
 \nonumber
 &v(t)=\sum^{\lceil \alpha_n \rceil-\lceil \alpha_{n-1} \rceil-1}_{i=0}\frac{t^i}{i!}u^{(i+\lceil \alpha_{n-1} \rceil)}_0
 +\frac {1}{\Gamma(\alpha_n-\lceil \alpha_{n-1} \rceil)}\int^t_0 {(t-s)}^{ \alpha_n-\lceil \alpha_{n-1} \rceil-1}\\
 \label{th2-1_eq2}
 &f\biggl (s,\int^s_0 \frac {{(s-w)}^{\lceil \alpha_{n-1} \rceil-\alpha_1-1}}{\Gamma(\lceil \alpha_{n-1} \rceil-\alpha_1)}v(w)dw,
  ,\cdots,\int^s_0 \frac {{(s-w)}^{\lceil \alpha_{n-1} \rceil-\alpha_{n-1}-1}}{\Gamma(\lceil \alpha_{n-1} \rceil-\alpha_{n-1})}v(w)dw
 \biggr)ds.
 \end{align}
 \end{theorem}

 \begin{proof}
 By Lemma \ref{le2-1}, we have
 \begin{equation}\label{th2-1_eq3}
 D^{\alpha_n-\lceil \alpha_{n-1} \rceil}u^{(\lceil \alpha_{n-1} \rceil)}(t)=D^{\alpha_n}u(t)=f(t, u(t), D^{\alpha_1}u(t),\cdots, D^{\alpha_{n-1}}u(t)).
 \end{equation}
 Applying the Riemann-Liouville integral $ J^{\alpha_n-\lceil \alpha_{n-1} \rceil} $ for both sides of (\ref{th2-1_eq3}) and making the substitution $ v(t)= u^{(\lceil \alpha_{n-1} \rceil)}(t) $, we obtain (\ref{th2-1_eq1}) and (\ref{th2-1_eq2}).
 In order to prove the converse, let $ v \in C[0,I] $ be a solution of (\ref{th2-1_eq2}).
 By (\ref{th2-1_eq1}), it is easy to see that $ u^{(\lceil \alpha_{n-1} \rceil)} (t)=v(t) $ and $  u^{(j)}(0)=u^{(j)}_0 $ for $ j=0,1,...,\lceil \alpha_{n-1} \rceil $.\\
 Since $ \lceil \alpha_{n-1} \rceil + 1 < \alpha_n $, by (\ref{th2-1_eq2}), we can easily prove that $ v(t) \in C^1 [0,I] $.
 Differentiating (\ref{th2-1_eq2}), we have
 \begin{align}
 \nonumber
 D^j v(t)&=\sum^{\lceil \alpha_n \rceil-\lceil \alpha_{n-1} \rceil-j-1}_{i=0}\frac{t^i}{i!}u^{(j+i+\lceil \alpha_{n-1} \rceil)}_0+\prod^j_{l=1} (\alpha_n-\lceil \alpha_{n-1} \rceil-l) \int^t_0 \frac {{(t-s)}^{ \alpha_n-\lceil \alpha_{n-1} \rceil-1-j}}{\Gamma(\alpha_n-\lceil \alpha_{n-1} \rceil)}\\
 \nonumber
 & f\biggl (s,\int^s_0 \frac {{(s-w)}^{\lceil \alpha_{n-1} \rceil-\alpha_1-1}}{\Gamma(\lceil \alpha_{n-1} \rceil-\alpha_1)}v(w)dw, \cdots,\int^s_0 \frac {{(s-w)}^{\lceil \alpha_{n-1} \rceil-\alpha_{n-1}-1}}{\Gamma(\lceil \alpha_{n-1} \rceil-\alpha_{n-1})}v(w)dw \biggr)ds
 \nonumber
 \end{align}
 and $ D^j v(0)=D^{j+\lceil \alpha_{n-1} \rceil}u(0)=u^{(j+\lceil \alpha_{n-1} \rceil)}_0 $
 for $ j=0,1,...,\lceil \alpha_n \rceil-\lceil \alpha_{n-1} \rceil-1 $.
 \begin{equation}
 \nonumber
 D^{\lceil \alpha_n \rceil-\lceil \alpha_{n-1} \rceil-1} v(t)=u^{(\lceil \alpha_n \rceil-1)}_0 + \prod^{\lceil \alpha_n \rceil-\lceil \alpha_{n-1} \rceil-1}_{l=1} (\alpha_n-\lceil \alpha_{n-1} \rceil-l) \int^t_0 \frac {{(t-s)}^{ \alpha_n-\lceil \alpha_n \rceil}} {\Gamma(\alpha_n-\lceil \alpha_{n-1} \rceil)}\\
 \end{equation}
 \begin{equation}
 \nonumber
 f\biggl (s,\int^s_0 \frac {{(s-w)}^{\lceil \alpha_{n-1} \rceil-\alpha_1-1}}{\Gamma(\lceil \alpha_{n-1} \rceil-\alpha_1)}v(w)dw,
  ,\cdots,\int^s_0 \frac {{(s-w)}^{\lceil \alpha_{n-1} \rceil-\alpha_{n-1}-1}}{\Gamma(\lceil \alpha_{n-1} \rceil-\alpha_{n-1})}v(w)dw
 \biggr)ds.
 \end{equation}
 Since $v(t) \in C^1 [0,1]$ and $u^{(\lceil \alpha_{n-1} \rceil)}_0=0$, by Lemma \ref{le2-2},
 \begin{equation}
 \nonumber
 f\biggl (s,\int^s_0 \frac {{(s-w)}^{\lceil \alpha_{n-1} \rceil-\alpha_1-1}}{\Gamma(\lceil \alpha_{n-1} \rceil-\alpha_1)}v(w)dw,
  ,\cdots,\int^s_0 \frac {{(s-w)}^{\lceil \alpha_{n-1} \rceil-\alpha_{n-1}-1}}{\Gamma(\lceil \alpha_{n-1} \rceil-\alpha_{n-1})}v(w)dw \biggr)
  \in C^1[0,I].
 \end{equation}
 Since $ f(0,0)=0 $, by Lemma \ref{le2-2}, $ D^{\lceil \alpha_n \rceil-\lceil \alpha_{n-1} \rceil-1} v(t) \in C^1[0,I] $.
 Thus $v \in C^{\lceil \alpha_n \rceil-\lceil \alpha_{n-1} \rceil}[0,I]$ and
 $u \in C^{\lceil \alpha_n \rceil}[0,I]$.
 \end{proof}

 \begin{theorem}\label{th2-2}
 Let $I>0, \alpha_{n-1} \in N , \alpha_n \notin N $ and $ \alpha_{n-1} +1 < \alpha_n $.
 Suppose that $ f:[0,I]\times R \rightarrow R $ is continuously differentiable  and $ f(0,u^{(\alpha_{n-1})}_0)=0 $.
 Then a function $ u \in C^{\lceil \alpha_n \rceil}[0,I] $ is a solution of the initial value problem (\ref{multi_term_equation})-(\ref{multi_term_condition}) if and only if
 \begin{equation}
 \nonumber
 u(t)=\sum^{ \alpha_{n-1} -1}_{i=0}\frac{t^i}{i!}u^{(i)}_0
 +\int^t_0 \frac{{(t-s)}^{ \alpha_{n-1}-1}}{(\alpha_{n-1} -1)!}v(s)ds, t \in [0,I],
 \end{equation}
 where $ v \in C[0,I] $ is a solution of the integral equation
 \begin{align}
 \nonumber
 &v(t)=\sum^{\lceil \alpha_n \rceil-\alpha_{n-1}-1}_{i=0}\frac{t^i}{i!}u^{(i+\alpha_{n-1})}_0
 +\int^t_0 \frac {{(t-s)}^{ \alpha_n-\alpha_{n-1}-1}}{\Gamma(\alpha_n-\alpha_{n-1})}\\
 \nonumber
 &f\biggl (s,\int^s_0 \frac {{(s-w)}^{ \alpha_{n-1}-\alpha_1-1}}{\Gamma(\alpha_{n-1}-\alpha_1)}v(w)dw,
 \cdots,\int^s_0 \frac {{(s-w)}^{ \alpha_{n-1}-\alpha_{n-2}-1}}{\Gamma(\alpha_{n-1}-\alpha_{n-2})}v(w)dw, v(s) \biggr)ds.
 \end{align}
 \end{theorem}

 \begin{proof}
 Similar to the proof of Theorem \ref{th2-1}, we can prove this result.
 \end{proof}

 \begin{theorem}\label{th2-3}
 Let $I>0,  \alpha_{n-1} \notin N, \alpha_n \in N $ and $ \lceil \alpha_{n-1} \rceil+1 \leq \alpha_n $.
 Suppose that $ f:[0,I]\times R \rightarrow R $ is continuous.
 Then a function $ u \in C^{ \alpha_n }[0,I] $ is a solution of the initial value problem (\ref{multi_term_equation})-(\ref{multi_term_condition}) if and only if
 \begin{equation}
 \nonumber
 u(t)=\sum^{ \lceil \alpha_{n-1} \rceil -1}_{i=0}\frac{t^i}{i!}u^{(i)}_0
 +\int^t_0 \frac{{(t-s)}^{ \lceil \alpha_{n-1} \rceil-1}}{(\lceil \alpha_{n-1} \rceil -1)!}v(s)ds, t \in [0,I],
 \end{equation}
 where $ v \in C[0,I] $ is a solution of the integral equation
 \begin{align}
 \nonumber
 &v(t)=\sum^{\alpha_n -\lceil \alpha_{n-1} \rceil-1}_{i=0}\frac{t^i}{i!}u^{(i+\lceil \alpha_{n-1} \rceil)}_0+\int^t_0 \frac {{(t-s)}^{ \alpha_n-\lceil \alpha_{n-1} \rceil-1}}{\Gamma(\alpha_n-\lceil \alpha_{n-1} \rceil)}\\
 \nonumber
 &f\biggl (s,\int^s_0 \frac {{(s-w)}^{\lceil \alpha_{n-1} \rceil-\alpha_1-1}}{\Gamma(\lceil \alpha_{n-1} \rceil-\alpha_1)}v(w)dw,
  ,\cdots,\int^s_0 \frac {{(s-w)}^{\lceil \alpha_{n-1} \rceil-\alpha_{n-1}-1}}{\Gamma(\lceil \alpha_{n-1} \rceil-\alpha_{n-1})}v(w)dw
 \biggr)ds.
 \end{align}
 \end{theorem}

 \begin{proof}
 Similar to the proof of Theorem \ref{th2-1}, we can prove this result.
 \end{proof}
 We can easily see that Lemma \ref{incorrect_lemma1} and Lemma \ref{incorrect_lemma2} are more general than Theorem \ref{th2-1} in the case $ n=2 $.  By making counterexamples, we show that Lemma \ref{incorrect_lemma1} and Lemma \ref{incorrect_lemma2} are incorrect.
 Firstly we present a counterexample of Lemma \ref{incorrect_lemma1}.
 Set $I=1, \alpha_1=1.8, \alpha_2=2.2, u_0^{(i)}=0,i=0,1,2 $ and
 \begin{equation}
 \nonumber
  f(t,s)= \frac{\Gamma(1.6)}{2\Gamma(1.4)}t^{0.4}+\frac{[\Gamma(1.6)\Gamma(1.8)]^{0.5}}{2\Gamma(1.4)}s^{0.5}.
 \end{equation}
 Then it is easy to see that $ v(t)=t^{0.6} $ is the solution of the equation(\ref{incorrect_lemma1_eq2}).
 By the equation (\ref{incorrect_lemma1_eq1}), we have
 \begin{equation}
 \nonumber
 u(t)=\frac{\Gamma(1.6)}{\Gamma(3.6)}t^{2.6}.
 \end{equation}
 It is clear that $ u \notin C^3[0,1]  $. Thus Lemma \ref{incorrect_lemma1} is incorrect.
 Secondly we give a counterexample of Lemma \ref{incorrect_lemma2}.
 Set $I=1,\alpha_1=1.4, \alpha_2=1.5, u_0^{(i)}=0,i=0,1 $ and
 \begin{equation}
 \nonumber
  f(t,s)= \frac{\Gamma(1.2)}{2\Gamma(1.1)}(t^{0.1}+s^{0.5}).
 \end{equation}
 Then we can easily see that $ v(t)=t^{0.2} $ is the solution of (\ref{incorrect_lemma2_eq2}).
 By (\ref{incorrect_lemma2_eq1}), we have
 \begin{equation}
 \nonumber
 u(t)=\frac{\Gamma(1.2)}{\Gamma(2.6)}t^{1.6}.
 \end{equation}
 It is evident that $ u \notin C^2[0,1]  $. Thus Lemma \ref{incorrect_lemma2} is incorrect.


 \section{Existence of solutions}\label{sec:3}

 In this section the existence of smooth solutions of the two-term Caputo-type fractional differential equation (\ref{two_term_equation})-(\ref{two_term_condition}) is discussed.
 In order to avoid the repetition of the proof process of our theorems, we state the following lemma.
 \begin{lemma}\label{le3-1}
 Let $ I>0 $, $ B $  be a convex bounded closed subset of $ C[0,I] $ and $ T $ is defined by
 \begin{equation}
 \nonumber
 Tv(t)=P(t)+\int^t_0 \frac {{(t-s)}^{\beta-1}}{\Gamma(\beta)}f\biggl (s,\frac {1}{\Gamma(\gamma)}\int^s_0 {(s-w)}^{\gamma-1}v(w)dw \biggr)ds.
 \end{equation}
 where $ \beta, \gamma>0 $ and $ P(t):[0,I] \rightarrow R, f(t,s):[0,I] \times R \rightarrow R$ are continuous.
 If $ T(B) \subset B $,  then $ T $ has at least one fixed point in $ B $.
 \end{lemma}
 \begin{proof}
 Similar to the proof of Theorem 1.2 in \cite{Deng2}, we can prove this result by using Schauder fixed point theorem.
 \end{proof}
 With the help of Lemma \ref{le3-1}, the initial value problem (\ref{two_term_equation})-(\ref{two_term_condition}) is reduced to the problem for finding a bounded, convex and closed subset $ B $ in $ Y $ such that $ T(B) \subset B $.
 \begin{theorem}\label{th3-1}
 Let $ k,I>0 $ and suppose that the hypotheses of Theorem \ref{th2-1} hold.
 Define
 \begin{equation}
 \nonumber
 G=\Biggl \{(t,v):t \in [0,I],|v| \leq \frac{I^{\lceil \alpha_1 \rceil-\alpha_1}}{\Gamma(\lceil \alpha_1 \rceil-\alpha_1+1)} \Biggl ( k + \sum^{\lceil \alpha_2 \rceil-\lceil \alpha_1 \rceil-1}_{i=0}\frac{I^i}{i!}\biggl |u^{(i+\lceil \alpha_1 \rceil)}_0 \biggr | \Biggr )   \Biggr \}
 \end{equation}
 and $ M:=sup_{(t,v) \in G}f(t,v) $.
 Then the initial value problem (\ref{two_term_equation})-(\ref{two_term_condition}) has at least one solution in $ C^{\lceil \alpha_2 \rceil}[0,h] $,
 where $ h $ is defined by
 \begin{equation}
 \nonumber
 h:=\left\{
    \begin{aligned}
            & I && \text{if $ M=0 $} \\
            & min \biggl \{ I,\bigl( {k \Gamma(\lceil \alpha_2 \rceil-\lceil \alpha_1 \rceil + 1)}/{M} \bigr) ^{\frac{1}{\lceil \alpha_2 \rceil-\lceil \alpha_1 \rceil}} \biggr \}&& \text{else}.\\
    \end{aligned}
    \right.
 \end{equation}
 \end{theorem}
 \begin{proof}
 We introduce the function $ P $ and the set $ B $ defined by
 \begin{equation}
 \nonumber
 P(t):=\sum^{\lceil \alpha_2 \rceil-\lceil \alpha_1 \rceil-1}_{i=0}\frac{t^i}{i!}u^{(i+\lceil \alpha_1 \rceil)}_0
 \end{equation}
 and $ B:=\{ v \in C[0,h]:\parallel P - v \parallel \leq k \} $, where $ \|\cdot\| $ is the Chebyshev norm.
 In order to prove our desired result, by Lemma \ref{le3-1} and Theorem \ref{th2-1}, we need to prove that $ T(B) \subset B $ where $ T $ is defined by
 \begin{align}
 \nonumber
 Tv(t)& =\sum^{\lceil \alpha_2 \rceil-\lceil \alpha_1 \rceil-1}_{i=0}\frac{t^i}{i!}u^{(i+\lceil \alpha_1 \rceil)}_0+\int^t_0 \frac {{(t-s)}^{ \alpha_2-\lceil \alpha_1 \rceil-1}}{\Gamma(\alpha_2-\lceil \alpha_1 \rceil)}f\biggl (s,\frac {1}{\Gamma(\lceil \alpha_1 \rceil-\alpha_1)}\\
 \nonumber
 &\int^s_0 {(s-w)}^{\lceil \alpha_1 \rceil-\alpha_1-1}v(w)dw \biggr)ds.
 \nonumber
 \end{align}
 For $ v \in B $ and $ t \in [0,h] $, we obtain
 \begin{equation}
 \nonumber
  | v(t) |\leq \parallel v\parallel  \leq k + \parallel P \parallel \leq k + \sum^{\lceil \alpha_2 \rceil-\lceil \alpha_1 \rceil-1}_{i=0}\frac{I^i}{i!}\biggl |u^{(i+\lceil \alpha_1 \rceil)}_0 \biggr |,
 \end{equation}
 so we have
 \begin{equation}
 \nonumber
 \int^t_0 \frac {{(t-s)}^{\lceil \alpha_1 \rceil-\alpha_1-1}|v(s)|}{\Gamma(\lceil \alpha_1 \rceil-\alpha_1)}ds \leq \frac{I^{\lceil \alpha_1 \rceil-\alpha_1}}{\Gamma(\lceil \alpha_1 \rceil-\alpha_1+1)} \Biggl ( k + \sum^{\lceil \alpha_2 \rceil-\lceil \alpha_1 \rceil-1}_{i=0}\frac{I^i}{i!}\biggl |u^{(i+\lceil \alpha_1 \rceil)}_0 \biggr | \Biggr )
 \end{equation}
 \begin{align}
 \nonumber
 |Tv(t)-P(t)| &\leq \int^t_0 \frac {{(t-s)}^{ \alpha_2-\lceil \alpha_1 \rceil-1}}{\Gamma(\alpha_2-\lceil \alpha_1 \rceil)} \Biggl | f\biggl (s,\int^s_0 \frac {{(s-w)}^{\lceil \alpha_1 \rceil-\alpha_1-1}v(w)}{\Gamma(\lceil \alpha_1 \rceil-\alpha_1)}dw \biggr) \Biggr | ds\\
 \nonumber
 &\leq \int^t_0 \frac {{M(t-s)}^{ \alpha_2-\lceil \alpha_1 \rceil-1}}{\Gamma(\alpha_2-\lceil \alpha_1 \rceil)} \leq \frac {{Mh}^{ \alpha_2-\lceil \alpha_1 \rceil}}{\Gamma(\alpha_2-\lceil \alpha_1 \rceil+1)} \leq k,
 \end{align}
 ,which implies that $ T(B) \subset B $.
 \end{proof}

 \begin{theorem}\label{th3-2}
 Let $ k>0 $ and suppose that the hypotheses of Theorem \ref{th2-2} hold.
 Define
 \begin{equation}
 \nonumber
 G:=\bigg \{(t,v):t \in [0,I],|v|\leq  k+\sum^{\lceil \alpha_2 \rceil- \alpha_1-1}_{i=0} \frac{I^i}{i!}\bigg | u^{(i+ \alpha_1 )}_0 \bigg |  \bigg \}
 \end{equation}
 and $ M:=sup_{(t,v) \in G}f(t,v) $.
 Then the initial value problem (\ref{two_term_equation})-(\ref{two_term_condition}) has at least one solution in $C^{\lceil \alpha_2 \rceil}[0,h] $,
 where $ h $ is defined by
 \begin{equation}
 \nonumber
 h:=\left\{
    \begin{aligned}
            & I && \text{if $ M=0 $} \\
            & min \biggl \{ I,\bigl( {k \Gamma(\lceil \alpha_2 \rceil- \alpha_1 + 1)}/{M} \bigr) ^{\frac{1}{\lceil \alpha_2 \rceil- \alpha_1 }} \biggr \}&& \text{else}.\\
    \end{aligned}
    \right.
 \end{equation}
 \end{theorem}
 \begin{proof}
 Similar to the proof of Theorem 6.1 in \cite{DieBoo}, we can prove this result.
 \end{proof}
 \begin{theorem}\label{th3-3}
 Let $ k>0 $ and suppose that the hypotheses of Theorem \ref{th2-3} hold.
 Define
 \begin{equation}
 \nonumber
 G=\Biggl \{(t,v):t \in [0,I],|v| \leq \frac{I^{\lceil \alpha_1 \rceil-\alpha_1}}{\Gamma(\lceil \alpha_1 \rceil-\alpha_1+1)} \Biggl ( k + \sum^{ \alpha_2 -\lceil \alpha_1 \rceil-1}_{i=0}\frac{I^i}{i!}\biggl |u^{(i+\lceil \alpha_1 \rceil)}_0 \biggr | \Biggr )   \Biggr \}
 \end{equation}
 and $ M:=sup_{(t,v) \in G}f(t,v) $.
 Then the initial value problem (\ref{two_term_equation})-(\ref{two_term_condition}) has at least one solution in $ C^{\lceil \alpha_2 \rceil}[0,h] $,
 where $ h $ is defined by
 \begin{equation}
 \nonumber
 h:=\left\{
    \begin{aligned}
            & I && \text{if $ M=0 $} \\
            & min \biggl \{ I,\bigl( {k \Gamma( \alpha_2 -\lceil \alpha_1 \rceil + 1)}/{M} \bigr) ^{\frac{1}{ \alpha_2 -\lceil \alpha_1 \rceil}} \biggr \}&& \text{else}.\\
    \end{aligned}
    \right.
 \end{equation}
 \end{theorem}
 \begin{proof}
 Similar to the proof of Theorem \ref{th3-1}, we can prove this result.
 \end{proof}

 \begin{remark}
 The existence of global solutions to the initial value problem (\ref{two_term_equation})-(\ref{two_term_condition}) can be obtained by
 \cite{Sin2}.
 \end{remark}

\begin{remark}
 The existence results for the two-term fractional differential equation (\ref{two_term_equation})-(\ref{two_term_condition}) presented in this section can be easily generalized to the multi-term fractional differential equation (\ref{multi_term_equation})-(\ref{multi_term_condition}).
 \end{remark}


\section*{References}

\end{document}